\documentclass[reqno,final]{amsart}
\usepackage{natbib}  
\usepackage{fancyhdr} 
\usepackage{color} 
\usepackage{hyperref} 
\usepackage{graphicx} 

\usepackage{pstricks}
\usepackage{amssymb}

\definecolor{aleacolor}{rgb}{0.16,0.59,0.78}
\hypersetup{
breaklinks,
colorlinks=true,
linkcolor=aleacolor,
urlcolor=aleacolor,
citecolor=aleacolor}

\renewcommand{\headheight}{11pt}
\renewcommand{\cite}{\citet}

\theoremstyle{plain}
\newtheorem{theorem}{Theorem}[section]                                          
                          
\newtheorem{lemma}[theorem]{Lemma}
\newtheorem{corollary}[theorem]{Corollary}

\theoremstyle{definition}
\newtheorem{definition}[theorem]{Definition}
\theoremstyle{remark}
\newtheorem{remark}[theorem]{Remark}
\newtheorem{example}[theorem]{Example}

\makeatletter \@addtoreset{equation}{section} \makeatother
\renewcommand\theequation{\thesection.\arabic{equation}}
\renewcommand\thefigure{\thesection.\arabic{figure}}
\renewcommand\thetable{\thesection.\arabic{table}}

\newcommand{\aleaIndex}[1]{\href{http://alea.impa.br/english/index_v#1.htm}{\bf #1}}
\eheader{Manuscript$\!\!$}{\aleaIndex{}}{2011}{1}{10}

\usepackage{amsmath, rotating, color}
\usepackage{amsfonts,amssymb, amsthm, mathrsfs}
\usepackage{bbm}
\usepackage{amsmath}
\usepackage[intlimits]{empheq}
\usepackage{multicol} 
\usepackage{amsthm}
\usepackage[notcite,notref]{showkeys} 

\newcommand{\smallx}{\mathpzc{x}}
\newcommand{\smally}{\mathpzc{y}}
\DeclareMathAlphabet{\mathpzc}{OT1}{pzc}{m}{it}

\begin{document}

\title[Compact metric measure spaces]{Compact metric measure spaces
  and $\Lambda$-coalescents coming down from infinity}

\author{Holger F. Biehler}
\author{Peter Pfaffelhuber}

\address{Albert-Ludwigs-Universit\"at Freiburg\newline Abteilung
  Mathematische Stochastik\newline Eckerstr.1,\newline D-79104
  Freiburg, Germany.}

\email{holger.biehler@web.de, p.p@stochastik.uni-freiburg.de}
\urladdr{\url{http://www.stochastik.uni-freiburg.de/~pfaffelh}}

\thanks{Research supported by the BMBF through FRISYS (Kennzeichen 0313921)}

\begin{abstract}
\noindent
We study topological properties of random metric spaces which arise by
$\Lambda$-coalescents. These are stochastic processes, which start
with an infinite number of lines and evolve through multiple mergers
in an exchangeable setting. We show that the resulting
$\Lambda$-coalescent measure tree is compact iff the
$\Lambda$-coalescent comes down from infinity, i.e.\ only consists of
finitely many lines at any positive time. If the $\Lambda$-coalescent
stays infinite, the resulting metric measure space is not even locally
compact.

Our results are based on general notions of compact and locally
compact (isometry classes of) metric measure spaces. In particular, we
give characterizations for general (random) metric measure spaces to
be (locally) compact using the Gromov-weak topology.
\end{abstract}

\maketitle

\section{Introduction}
Metric structures arise frequently in probability theory. Prominent
examples are random trees (e.g.\
\citealp{Ald1993,EvansOConnell1994,MR1714707,Berestycki01}), where the
distance between two points is given by the length of the shortest
path connecting the points. A class of random trees is given by
coalescent processes, where a subset of an infinite number of
\emph{lines} can merge and the distance of two leaves is proportional
to the coalescence time
(\citealp{Kingman1982a,Pit1999,Aldous1999,Schweinsberg2000,Evans2000}). The
complexity of this class of processes is properly described by the
concepts of $\Lambda$-coalescent, where any set of lines can merge to
a single line (a multiple collision, \citealp{Pit1999}) and
$\Xi$-coalescents, where any set of lines can merge to several lines
at the same time (a simultaneous multiple collision,
\citealp{Schweinsberg2000a}). The resulting metric space has so far
mostly been studied in the simplest case, where only binary mergers
are allowed, the Kingman-coalescent \citep{Kingman1982a, Evans2000}.

Analyzing metric structures requires geometrical and topological
foundations. In the context of Riemannian geometry, such foundations
have already been laid by Gromov, summarized in his book
(\citealp{Gromov}, see also \citealp{Ver1998,BurBurIva01}). These
authors study convergence of (isometry classes of) compact metric
spaces by the notion of Gromov-Hausdorff convergence. In addition,
Gromov introduced a topology on the space of (isometry classes of)
metric measure spaces (mm-spaces, for short), which are metric spaces
equipped with a measure. We will call this the Gromov-weak topology in
the sequel (see also \citealp{GPWmetric09}).

In probability theory, results on weak convergence and stochastic
process theory require that the underlying space is Polish. In
addition, a characterization of the compact sets is required in order
to show tightness. These concepts have been worked out based on
Gromov's notions by \cite{EvaPitWin2006} and \cite{GPWmetric09}.

~

The goal of the present paper is as follows: we concentrate on the
spaces of \emph{locally compact} and \emph{compact} mm-spaces and give
a characterization of these (see Theorems~\ref{T:compact} and
\ref{T:loccompact}). In addition, we apply these general results to
random mm-spaces ($\Lambda$-coalescent measure trees) which arise in
connection to $\Lambda$-coalescents. Recall that $\Lambda$-coalescents
fall into one of two categories, depending on $\Lambda$. Either a
$\Lambda$-coalescent comes down from infinity, meaning that it can be
started with an infinite number of lines and only finitely many are
left at any positive time, or it stays infinite for all times (see
\citealp[Proposition 23]{Pit1999}). The proof of the following result
is given in Section~\ref{S:4}.

\begin{theorem}[Coming down from infinity and compactness]\label{T1}
  Let $\Lambda$ be a finite measure on $[0,1]$ and $(\Pi_t)_{t\geq 0}$
  the corresponding $\Lambda$-coalescent. Moreover, $\mathcal L$ is
  the associated $\Lambda$-coalescent measure tree, taking values in
  the space of mm-spaces. Then the following is equivalent.
  \begin{enumerate}
  \item $(\Pi_t)_{t\geq 0}$ comes down from infinity, i.e.\
    $\#\Pi_t<\infty$ almost surely, for all $t>0$.
  \item $\mathcal L$ is compact, almost surely.
  \end{enumerate}
  If (1) (or 2) does not hold, $\mathcal L$ is not even locally
  compact.
\end{theorem}

\noindent
We proceed as follows: In Section~\ref{S:mm} we develop our general
theory on compact and locally compact isometry classes of metric
measure spaces. Section~\ref{S:lambda} contains a short introduction
to $\Lambda$-coalescent measure trees. Finally, the proof of
Theorem~\ref{T1} is given in Section~\ref{S:4}. We remark that the
application of (locally) compact mm-spaces is not restricted to
trees. For example, it is possible to study large random planar maps,
as given in \cite{LeGall2007}, or random Graphs (e.g.\ the
Erd\H{o}s-Renyi random graph, \citealp{ABBG09}), by our notions.

\section{Metric measure spaces}
\label{S:mm}
We start with some notation. Our main results, the characterization of
compact and locally compact mm-spaces, is given in
Theorems~\ref{T:compact} and~\ref{T:loccompact}.

\begin{remark}[Notation]
  As usual, given a topological space $(X,\mathscr{O}_X)$, we denote
  by $\mathscr{M}_1(X)$ the space of all probability measures on the
  Borel-$\sigma$-algebra $\mathscr{B}(X)$. The \emph{support} of $\mu
  \in \mathscr{M}_1(X)$, supp$(\mu)$, is the smallest closed set $X_0
  \subseteq X$ such that $\mu(X \setminus X_0) = 0$. The
  \emph{push-forward} of $\mu$ under a measurable map $\varphi$ from
  $X$ into another topological space, $(Z,\mathscr O_Z)$, is the
  probability measure $\varphi_*\mu \in \mathscr{M}_1(Z)$ defined for
  all $A \in \mathscr{B}(Z)$ by $\varphi_*\mu (A) :=
  \mu(\varphi^{-1}(A)).$ We denote weak convergence in
  $\mathscr{M}_1(X)$ by $\xRightarrow{}$.
\end{remark}

\begin{definition}[Metric measure and mm-spaces]\mbox{}
  \begin{enumerate}
  \item A \emph{metric measure space} is a triple $(X,r,\mu)$ such
    that $X\subseteq \mathbb R$, $(X,r)$ is a complete and separable
    metric space which is equipped with a probability measure $\mu$ on
    $\mathscr B(X)$.  We say that $(X,r,\mu)$ and $(X',r',\mu')$ are
    \emph{measure-preserving isometric} if there exists an isometry
    $\varphi$ between supp$(\mu)\subseteq X$ and supp$(\mu')\subseteq
    X'$ such that $\mu'|_{\text{supp}(\mu')} =
    \varphi_*(\mu|_{\text{supp}(\mu)})$. It is clear that the property
    of being measure-preserving isometric is an equivalence relation.
  \item The equivalence class of the metric measure space $(X,r,\mu)$
    is called the mm-space of $(X,r,\mu)$ and is denoted
    $\overline{(X,r,\mu)}$. The set of mm-spaces is denoted $\mathbb
    M$ and generic elements are $\smallx, \smally,...$
  \item An mm-space $\smallx\in\mathbb M$ is \emph{(locally) compact}
    if there is $(X,r,\mu)\in\smallx$ such that $(X,r)$ is (locally)
    compact. The space of (locally) compact mm-spaces is denoted
    $\mathbb M_c$ ($\mathbb M_{lc}$).
  \end{enumerate}
\end{definition}

\noindent
Following \cite{GPWmetric09}, we equip $\mathbb{M}$ with the
Gromov-weak topology as follows.

\begin{definition}[Gromov-weak topology]
  For a metric space $(X,r)$ define
  \begin{align*}
    R^{(X,r)}: \begin{cases}
      X^\mathbb{N} &\to \mathbb{R}_+^{\binom{\mathbb{N}}{2}} \\
      (x_i)_{i \in \mathbb{N}} &\mapsto (r(x_i,x_j))_{1 \leq i <
        j} \end{cases}
  \end{align*}
  the map which sends a sequence of points in $X$ to its distance
  matrix and for an mm-space $\smallx = \overline{(X,r,\mu)}$ we
  define the \emph{distance matrix distribution} by
  \begin{align*}
    \nu^{\smallx} := (R^{(X,r)})_*\mu^{\otimes \mathbb{N}} \in
    \mathscr{M}_1(\mathbb{R}_+^{\binom{\mathbb{N}}{2}}),
  \end{align*}
  where $\mu^{\otimes \mathbb{N}}$ is the infinite product measure of
  $\mu$, where $\mathbb R_+^{\binom{\mathbb N} {2}}$ is equipped with
  the product $\sigma$-field. We say that a sequence
  $\smallx_1,\smallx_2,...\in\mathbb M$ converges \emph{Gromov-weakly}
  to $\smallx\in\mathbb M$ if
  $$ \nu^{\smallx_n} \xRightarrow{n\to\infty} \nu^\smallx.$$
\end{definition}

\noindent
Note that $\nu^\smallx$ does not depend on the representative
$(X,r,\mu)\in\smallx$, hence is well-defined. 

\begin{remark}[When is a random mm-space compact?]
  Recall from Theorem~1 of \cite{GPWmetric09} that the space $\mathbb
  M$, equipped with the Gromov-weak topology, is Polish. Hence,
  $\mathbb M$ allows to use standard tools from probability, e.g.\
  from the theory of weak convergence.

  In order to show that a random variable taking values in $\mathbb M$
  is supported by the space of locally compact or compact mm-spaces,
  there are two strategies, formulated here in the case of compact
  mm-spaces:

  Either, consider the Gromov-weak topology on $\mathbb M_c$.
  Defining an approximating sequence in $\mathbb M_c$ and showing that
  the sequence is tight in $\mathbb M_c$ ensures compactness of the
  limiting object. Note that any mm-space can be approximated by
  finite (hence compact) mm-spaces, so $\mathbb M_c$ is not closed in
  $\mathbb M$. So, this approach amounts to knowing the compact sets
  in $\mathbb M_c$. See Proposition~6.2 of
  \cite{GrevenPfaffelhuberWinter2010} for an example.

  Our application to the $\Lambda$-coalescent measure tree in
  Section~\ref{S:4} relies on a different approach. It is possible to
  give handy characterizations of compact mm-spaces; see
  Theorem~\ref{T:compact}. Hence, if we are given a random variable
  taking values in $\mathbb M$ through a sequence of mm-spaces, it is
  possible to check directly if the limiting object is compact.
\end{remark}

\begin{definition}[Distance distribution, Moduli of mass
  distribution]\mbox{}
  Let \label{def:mod1} $\smallx \in \mathbb{M}$. We set
  $\underline{\underline r}:=(r_{ij})_{1\leq i<j}\in\mathbb
  R_+^{\binom{\mathbb N}{2}}$.
  \begin{itemize}
  \item[(a)] Let $r: \mathbb R^{\binom{\mathbb N}{2}}_+\to\mathbb R_+$
    be given by $r\big( \underline{\underline r}\big) :=
    r_{12}$. Then, the \emph{distance distribution} is given by
    $w_{\smallx}:= r_*\nu^\smallx$, i.e.,
    \begin{align*}
      w_{\smallx}(\cdot) := \nu^\smallx\big\{\underline{\underline r}:
      r_{12}\in \cdot\big\}.
    \end{align*}
  \item[(b)] For $\varepsilon>0$, define $s_\varepsilon: \mathbb
    R^{\binom{\mathbb N}{2}}_+\to\mathbb R_+$ by
    $$ s_\varepsilon\big(\underline{\underline r}\big) := \lim_{n\to\infty} \frac 1n \sum_{j=1}^n 
    1_{\{r_{1j}\leq \varepsilon\}}$$ if the limit exists (and zero
    otherwise). Note that $s_\varepsilon(\underline{\underline r})$
    exists for $\nu^\smallx$-almost all $\underline{\underline r}$ by
    exchangeability and de Finetti's Theorem.  For $\delta>0$, the
    \emph{moduli of mass distribution} are
    \begin{align*}
      v_\delta (\smallx) := \inf\{\varepsilon > 0 : ~
      \nu^\smallx\big\{\underline{\underline r}:
      s_\varepsilon(\underline{\underline r})\leq \delta\big\}\leq \varepsilon\}
    \end{align*}
    and
    \begin{align*}
      \widetilde v_\delta (\smallx) := \inf\{\varepsilon > 0 : ~
      \nu^\smallx\big\{\underline{\underline r}:
      s_\varepsilon(\underline{\underline r})\leq \delta\big\}=0\}.
    \end{align*}
  \end{itemize}
\end{definition}

\begin{example}[Representatives of $\smallx$]
  \label{rem:repres} Let $\smallx = \overline{(X,r,\mu)}$. Without
  loss of generality we assume that supp$(\mu)=X$. Since $\nu^\smallx
  = (R^{(X,r)})_\ast \mu^{\otimes\mathbb N}$, we have that
  $$ w_\smallx(\cdot) = \mu^{\otimes 2}\{(x,y): r(x,y)\in \cdot\}.$$
  Moreover, 
  \begin{align}\label{eq:repres1}
    \nu^\smallx\{\underline{\underline r}:
    s_\varepsilon(\underline{\underline r}) \in \cdot\} = \mu\{x:
    \mu(B_\varepsilon(x)) \in\cdot\}
  \end{align}
  by construction, where $B_\varepsilon(x)$ is the closed ball of
  radius $\varepsilon$ around $x$. This implies that
  $$ v_\delta(\smallx) \leq \varepsilon \qquad \iff \qquad \mu\{x: \mu(B_\varepsilon(x)) 
  \leq \delta\} \leq \varepsilon.$$ In particular,
  $v_\delta(\smallx)\leq \varepsilon$ means, that thin points (in the
  sense that $\mu(B_\varepsilon(x))\leq \delta$) are rare (i.e.\ carry
  mass at most $\varepsilon$). Moreover, 
  $$ \widetilde v_\delta(\smallx) \leq \varepsilon \qquad \iff \qquad \mu\{x: \mu(B_\varepsilon(x)) 
  \leq \delta\} = 0.$$ This means that there are $\mu$-almost surely
  no points which are too thin (in the sense that
  $\mu(B_\varepsilon(x))\leq \delta$). 
\end{example}

\begin{definition}[Size of $\varepsilon$-separated set]
  Let $\underline{\underline r}\in\mathbb R_+^{\binom{\mathbb
      N}{2}}$. For $\varepsilon>0$, define the \emph{maximal size of
    an $\varepsilon$-separated set} by
  $$\xi_\varepsilon(\underline{\underline r}) := \sup\big\{N\in\mathbb N: \;\exists 
  k_1<...<k_N: (r_{k_i,k_j})_{1\leq i<j\leq N} \in
  (\varepsilon,\infty)^{\binom{N}{2}}\big\}.$$
\end{definition}

\begin{lemma}[$\xi_\varepsilon$ is constant, $\nu^\smallx$-almost
  surely]\label{l4} Let $\smallx\in\mathbb N$ and
  $\varepsilon>0$. Then, $\xi_\varepsilon$ is constant,
  $\nu^\smallx$-almost surely and equals
  $$ \xi_\varepsilon(\smallx):=\inf\big\{N\in\mathbb N: \nu^\smallx
  \big( \rho_N^{-1} \big((\varepsilon,\infty)^{\binom{N}{2}}\big)
  >0\big\},$$ where $\rho_N: \mathbb R^{\binom{\mathbb N}{2}} \to
  \mathbb R^{\binom{N}{2}}$ is the projection on the first $\binom N
  2$ coordinates.
\end{lemma}

\begin{proof}
  Assume $\smallx = \overline{(X,r,\mu)}$. Let $x_1,x_2,...\in X$ be
  such that $\xi_\varepsilon((r(x_i, x_j))_{1\leq i<j}) = N$. Then,
  $N$ is the maximal size of an $\varepsilon$-separated set in
  $(X,r)$, $\mu^{\otimes \mathbb N}$-almost surely. All results
  follow, since $\nu^\smallx = (R^{(X,r)})_\ast \mu^{\otimes\mathbb
    N}$ and since $\nu^\smallx$ is exchangeable.
\end{proof}

\begin{remark}[Tightness in $\mathbb M$]
  Recall from Theorem~2 in \cite{GPWmetric09} that for any
  $\smallx\in\mathbb M$, it holds that $v_\delta(\smallx)
  \xrightarrow{\delta\to 0}0$. Moreover, a set $\Gamma\subseteq
  \mathbb M$ is pre-compact iff $\{w_\smallx: \smallx\in \Gamma\}$ is
  tight (as a family in $\mathscr M_1(\mathbb R_+)$) and
  $\sup_{\smallx\in\Gamma}v_\delta(\smallx) \xrightarrow{\delta\to
    0}0$.

  \sloppy This leads to a characterization of tightness for a family
  of random mm-spaces, see \cite{GPWmetric09}, Theorem 3: Here, (the
  distributions of) a family $\{\mathcal X: \mathcal X\in\Gamma\}$ of
  $\mathbb M$-valued random variables is tight iff $\{\langle
  w_{\mathcal X}\rangle: \mathcal X\in\Gamma\}$ is tight (where
  $\langle w_{\mathcal X}\rangle$ is the first moment measure of
  $(w_{\mathcal X})_\ast \mathbf P \in \mathscr M_1(\mathscr
  M_1(\mathbb R_+))$ and $\sup_{\mathcal X\in\Gamma} \mathbf
  E[v_\delta(\mathcal X)] \xrightarrow{\delta\to 0} 0.$ Given a
  sequence of random mm-spaces, we can use these results in order to
  obtain limiting objects, at least along subsequences.
\end{remark}

\noindent
Now we come to a characterization of compact mm-spaces. 

\begin{theorem}[Compact mm-spaces] \label{T:compact} Let $\smallx
  \in\mathbb M$. The following conditions are equivalent.
  \begin{enumerate}
  \item The mm-space $\smallx$ is compact, i.e.\ $\smallx\in\mathbb
    M_c$.
  \item For all $\varepsilon>0$, it holds that $\xi_\varepsilon(\smallx)<\infty$.
  \item For all $\varepsilon>0$, there is $\delta > 0$ such that
    $\widetilde v_\delta(\smallx)\leq \varepsilon$.
  \end{enumerate}
\end{theorem}

\noindent
The following characterization of random, almost surely compact
mm-spaces is immediate.

\begin{corollary}[Random compact mm-spaces]\label{cor:compact}
  Let \label{cor:cp} $\mathcal X$ be a random variable taking values
  in $\mathbb M$. The following conditions are equivalent.
  \begin{enumerate}
  \item The mm-space $\mathcal X$ is compact, almost surely, i.e.\
    $\mathbf P(\mathcal X\in\mathbb M_c)=1$.
  \item For all $\varepsilon>0$, it holds that $\mathbf
    P(\xi_\varepsilon(\mathcal X)<\infty)=1$.
  \item For all $\varepsilon>0$, there is a random variable $\Delta>0$
    with $\mathbf P(\widetilde v_\Delta(\mathcal
    X)\leq\varepsilon)=1$.
  \end{enumerate}
\end{corollary}

\begin{remark}[Size of $\varepsilon$-separated set and size of
  $\varepsilon$-covering]
  The following observation will be used in the proof of
  Theorem~\ref{T:compact}: Let $(X,r)$ be a metric space and
  $\varepsilon>0$, let $\xi_\varepsilon$ be the maximal size of an
  $\varepsilon$-separated set 
  and $N_\varepsilon$ be the minimal number of $\varepsilon$-balls
  needed to cover $(X,r)$. Then
  $$ N_{\varepsilon} \leq \xi_\varepsilon\leq N_{\varepsilon/2}.$$

  \noindent
  In order to see this, let $x_1,...,x_{\xi_\varepsilon}$ be a maximal
  $\varepsilon$-separated set. Then, $X =
  \bigcup_{i=1}^{\xi_\varepsilon} B_\varepsilon(x_i)$, since
  otherwise, we find $x \in
  X\setminus\Big(\bigcup_{i=1}^{\xi_\varepsilon}
  B_\varepsilon(x_i)\Big)$ and hence, the set is not maximal. This
  shows $N_\varepsilon\leq \xi_\varepsilon$. For the second
  inequality, it is clear that
  $B_{\varepsilon/2}(x_1),...,B_{\varepsilon/2}(x_{\xi_\varepsilon})$
  are disjoint. Hence, any set of centers of $\varepsilon/2$-balls
  which cover $(X,r)$ must hit each $B_{\varepsilon/2}(x_i)$ at least
  once. As a consequence, $\xi_\varepsilon\leq N_{\varepsilon/2}$.
\end{remark}

\begin{proof}[Proof of Theorem~\ref{T:compact}] Let $\smallx =
  \overline{(X,r,\mu)}$. We use the notation laid out in
  Remark~\ref{rem:repres}. In particular, recall \eqref{eq:repres1}.

  $(1)\Rightarrow (2)$: Let $\smallx$ be compact and $\varepsilon >
  0$. Then $(X,r)$ is totally bounded and there is $N_{\varepsilon/2}
  \in \mathbb{N}$ such that $(X,r)$ can be covered by
  $N_{\varepsilon/2}$ balls of radius $\varepsilon/2$. Then we find
  $\xi_\varepsilon(\smallx) \leq N_{\varepsilon/2} <\infty$ by the
  last remark.

  $(2)\Rightarrow (3)$: Let $\varepsilon>0$. The space $(X,r)$ can be
  covered by $\xi_{\varepsilon/2}(\smallx)<\infty$ balls of radius
  $\varepsilon/2$, again by the last remark. Let $x_1, \dots ,
  x_{\xi_{\varepsilon/2}}$ be centers of such balls and $\delta :=
  \min \{ \mu(B_{\varepsilon/2} (x_i)) : \mu(B_{\varepsilon/2} (x_i))
  > 0\}$. Then $\delta > 0$. Now take any $x \in X$ and choose $i \in
  \{1, \dots , \xi_{\varepsilon/2} \}$ such that $x \in
  B_{\varepsilon/2} (x_i)$. Then we have
  \begin{align*}
    \mu(B_{\varepsilon} (x)) \geq \mu(B_{\varepsilon/2} (x_i)) \geq
    \delta.
  \end{align*}
  Hence,
  $$ \nu^{\smallx}\{\underline{\underline r}:
  s_{\varepsilon}(\underline{\underline r})\leq\delta\} = \mu\{x\in X:
  \mu(B_{\varepsilon}(x))\leq \delta\} = 0.  $$

  \noindent
  $(3)\Rightarrow (1)$: It suffices to show that $(X,r)$ is totally
  bounded. Let $\varepsilon > 0$. By assumption, there is $\delta > 0$
  such that $$\nu^\smallx\{\underline{\underline r}:
  s_\varepsilon(\underline{\underline r})\leq\delta\}=\mu\{x\in X:
  \mu(B_\varepsilon(x))\leq \delta\} = 0.$$ We show that there is a
  finite maximal $2\varepsilon$-separated set in $X$. For this, take a
  maximal $2\varepsilon$-separated set $S\subseteq X$ (and without
  loss of generality assume that supp$(\mu) = X$). Then, using the
  last remark,
  \begin{align*}
    1 = \mu(X) = \mu\Big(\bigcup_{x\in S} B_{2\varepsilon} (x)\Big)
    \geq \mu\Big(\bigcup_{x\in S} B_\varepsilon (x)\Big) = \sum_{x\in
      S} \mu(B_\varepsilon (x)) \geq |S| \cdot \delta,
  \end{align*}
  since $\mu(B_\varepsilon (x))>\delta$ holds $\mu$-almost surely by
  assumption. Now, $|S|\leq 1/\delta<\infty$ and $\varepsilon>0$ was
  arbitrary, so $(X,r)$ is totally bounded.
\end{proof}

\noindent
Next, we come to a characterization of locally compact
mm-spaces. Again some notation is needed.

\begin{definition}[$\delta$-restriction] Let
  $\underline{\underline r} := (r_{ij})_{1\leq i<j} \in \mathbb
  R_+^{\binom{\mathbb N}{2}}$. Set $\widehat\tau_\delta(0):=1$
  and
  $$ \widehat\tau_\delta(i+1) := \inf\{j>\widehat\tau_\delta(i): r_{1j}\leq\delta\}.$$  
  Then, $$\tau_\delta(\underline{\underline r}) :=
  (r_{\widehat\tau_\delta(i),
    \widehat\tau_\delta(j)})_{1\leq i<j}$$ is called the
  \emph{$\delta$-restriction} of $\underline{\underline r}$.
\end{definition}

\begin{remark}[$\delta$-restriction for distance matrices.]
  \sloppy Let $\smallx = \overline{(X,r,\mu)} \in \mathbb M$ and
  $x_1,x_2,... \in X$. We note that $x_k \in
  \widehat\tau_\delta(\mathbb N)$ iff $r(x_1,x_k)\leq
  \delta$. Hence, $\tau_\delta((r(x_i,x_j))_{1\leq i\leq
    j})$ is the distance matrix distribution for points among
  $x_2,x_3,...$ which have distance at most $\delta$ to
  $x_1$. So,
  \begin{align*}
    (\tau_\delta)_\ast \nu^\smallx(\cdot) & =
    \nu^\smallx\{\tau_\delta(\underline{\underline r}) \in
    \cdot\} = \nu^\smallx\{\underline{\underline r} \in \cdot |
    r_{12}, r_{13},...\leq\delta\} \\ & = \mu^{\otimes \mathbb
      N}\{(r(x_i, x_j))_{1\leq i < j} \in \cdot | r(x_1,x_j)\leq
    \delta \text{ for all }j=2,3,...\}
  \end{align*}
  Clearly, $(\tau_\delta)_\ast \nu^\smallx$ is exchangeable,
  since $\nu^\smallx$ is exchangeable.
\end{remark}

\begin{theorem}[Locally compact mm-spaces]
  \label{T:loccompact}
  Let $\smallx \in\mathbb M$. The following conditions are equivalent.
  \begin{enumerate}
  \item The mm-space $\smallx$ is locally compact, $\smallx\in\mathbb M_{lc}$.
  \item It holds that
    $$ \nu^\smallx\Big(\bigcap_{0<\eta<\delta} 
    \Big\{\underline{\underline r}: \xi_\eta(
    \tau_\delta(\underline{\underline
      r}))<\infty\Big\}\Big)\xrightarrow{\delta\to 0}1.$$
  \end{enumerate}
\end{theorem}

\begin{proof}
  Let $\smallx = \overline{(X,r,\mu)}$. Then, $\smallx$ is locally
  compact iff for $\mu$-almost all $x\in X$ there is $\delta>0$, such
  that the ball $B_\delta(x)$ can be covered by a finite number of
  balls with radius $\eta$, for all $0<\eta<\delta$. Hence, 
  \begin{align*}
    1 
    & = \mu\Big(\bigcup_{\delta>0}
    \bigcap_{0<\eta<\delta}\big\{x: B_\varepsilon(x) \text{ can
      be covered by finitely many balls of radius }\eta\big\}\Big) \\
    & = \lim_{\delta\to 0} \mu\Big( \bigcap_{0<\eta<\delta}\big\{x:
    \text{ the maximal $\eta$-separated set in $B_\delta(x)$ is finite}\big\}\Big)\\
    & = \lim_{\delta\to 0} \mu^{\otimes\mathbb N}\Big(
    \bigcap_{0<\eta<\delta}\big\{(x_1,x_2,...):
    \xi_\eta((r_{x_i,x_j})_{2\leq i<j})<\infty | r(x_1,x_2), r(x_1,
    x_3),...<\delta\big\}\Big)
    \\
    & = \lim_{\delta\to 0} \mu^{\otimes\mathbb N}\Big(
    \bigcap_{0<\eta<\delta}\big\{(x_1,x_2,...):
    \xi_\eta(\tau_\delta((r_{x_i,x_j})_{1\leq i<j}))<\infty\}\Big) \\
    & = \lim_{\delta\to 0} \nu^\smallx\Big(\bigcap_{0<\eta<\delta}
    \Big\{\underline{\underline r}: \xi_\eta(
    \tau_\delta(\underline{\underline r}))<\infty\Big\}\Big).
  \end{align*}
\end{proof}

\section{$\Lambda$-coalescents}
\label{S:lambda}
\noindent
We come to the application of the general results from the last
section to metric spaces which arise in the context of coalescents
which allow for multiple mergers. The proof of Theorem~\ref{T1} is
given in the next section.  Introduced by \cite{Pit1999},
$\Lambda$-coalescents are usually described by Markov processes taking
values in partitions of $\mathbb N$, which become coarser as time
evolves, almost surely, and are exchangeable. More exactly, we define
$(\Pi_t)_{t\geq 0} = (\Pi^\Lambda_t)_{t\geq 0}$, starting in the
trivial partition of $\mathbb N$. For a finite measure $\Lambda$ on
$[0,1]$, set
\begin{align}
  \label{eq:lambdabk}
  \lambda_{b,k} = \int_0^1 x^{k-2}(1-x)^{b-k} \Lambda(dx).
\end{align}
Among any set of $b$ partition elements in $\Pi_t$, each subset of
size $k$ merges to one partition element at rate $\lambda_{b,k}$. It
is easy to check that such a process is well-defined (i.e.\ the
$\lambda_{b,k}$'s are consistent) and leads to an exchangeable
partition of $\mathbb N$ for all $t\geq 0$. In our analysis we
restrict ourselves to measures $\Lambda$ which do not have an atom at
1; see Example~20 in \cite{Pit1999} for a discussion of this case.

One intuitive way to construct a $\Lambda$-coalescent (given $\Lambda$
has no atom at 0) is as follows: consider a Poisson-process with
intensity measure $\frac{\Lambda(dx)}{x^2} \cdot dt$ on $[0,1] \times
\mathbb R_+$. At any Poisson point $(x,t)$, mark all partition
elements, which are available by time $t$ with probability $x$ and
merge all marked partition elements.

\bigskip

The set of $\Lambda$-coalescents falls into (at least) three
classes. The class of $\Lambda$-coalescents coming down from infinity
(see Property 1 in Theorem~\ref{T1}), the larger class of processes
having the \emph{dust-free}-property, i.e.\ $\{f(\Pi_t^1)>0$ for all
$t>0$, almost surely, where $f(\Pi_t^j)$ is the frequency of the
partition element containing $j$ at time $t$, $j\in\mathbb N$). All
other $\Lambda$-coalescents contain \emph{dust}, which is a positive
frequency of natural numbers forming their own partition element.

Starting with \cite{Schweinsberg2000}, sharp conditions for a
$\Lambda$-coalescents coming down from infinity have been
given. Precisely, it was stated that a $\Lambda$-coalescent comes down
from infinity iff
\begin{align} \label{EquaGamma<Infty} \sum_{b=2}^{\infty}
  \Big(\sum_{k=2}^{b} k \binom{b}{k} \lambda_{b,k}\Big)^{-1} < \infty.
\end{align}
It has been shown by \cite{BertoinLeGall2006} that this is equivalent
to
\begin{align*}
\int_t^\infty \psi(q)^{-1} dq < \infty.
\end{align*}
for some $t>0$ where
\begin{align*}
\psi(q) = \int_0^1 (e^{-qx} - 1 + qx) x^{-2} \Lambda(dx).
\end{align*}
The larger class of coalescents having the dust-free property is
characterized by the requirement that
\begin{align}\label{eq:df}
  \int_0^1 x^{-1} \Lambda(dx) = \infty,
\end{align}
see Theorem 8 in \cite{Pit1999}.

\bigskip

\noindent
Let $\Pi^\Lambda := \Pi = (\Pi_t : t \geq 0)$ be the
$\Lambda$-coalescent. Then for almost all sample paths of
$\Pi^\Lambda$, there is a metric $r^\Pi$ on $\mathbb{N}$, associated
to $\Pi$, defined by
\begin{align*}
  r^\Pi(i,j) := \inf \{t \geq 0 : \text{$i,j$ in the same partition
    element of $\Pi_t$}\},
\end{align*}
that is the time needed for $i$ and $j$ to coalesce. We denote by
$(L^\Pi,r^\Pi)$ the completion of $(\mathbb{N},r^\Pi)$.  In order to
equip $(L^\Pi,r^\Pi)$ with a probability measure, we use a limit
procedure. Set
$$ H^n(\Pi) := \overline{\Big(L^\Pi, r^\Pi, \tfrac 1n \sum_{i=1}^n \delta_i\Big)}$$
Then, the family of $\mathbb M$-valued random variables
$(H^n(\Pi))_{n=1,2,...}$ converges in distribution with respect to the
Gromov-weak topology iff $\Pi^\Lambda$ is dust-free, i.e.\
\eqref{eq:df} holds (see Theorem 5 in \citealp{GPWmetric09}). Since
coalescent processes are associated with tree-like structures, we call
the limiting mm-space $\mathcal L = \overline{(L^\Pi, r^\Pi,
  \mu^\Pi)}$ the \emph{$\Lambda$-coalescent measure tree}.

\section{Proof of Theorem \ref{T1}}
\label{S:4} Let $N(t) := \#\Pi_t$ denote the number of blocks in the
partition $\Pi_t$ and note that $\xi_{\varepsilon}(\mathcal L) \leq
N(\varepsilon)$ where $\xi_{\varepsilon}(\mathcal L) < N(\varepsilon)$
is only possible if there are partition elements in $\Pi_\varepsilon$
which carry no mass in $\mathcal L$.

$(1)\Rightarrow (2)$: Using Corollary~\ref{cor:cp}, we must show that
for all $\varepsilon>0$, we have $\xi_{\varepsilon}(\mathcal
L)<\infty$ almost surely. This follows directly from the fact that
$\xi_{\varepsilon}(\mathcal L) \leq N(\varepsilon)$ and the assumption
that $\Pi$ comes down from infinity.

$(2)\Rightarrow (1)$: The proof is by contradiction. Assume
$\mathcal{L}$ is compact and $\Pi$ stays infinite for some time
$\varepsilon>0$. Since $\Pi_\varepsilon$ contains no dust, we have
that $f((\Pi_\varepsilon^j))>0$ for all $j=1,2,...$, almost
surely. Since there are infinitely many lines up to time
$\varepsilon$, we find partition elements of arbitrarily small mass.
This implies that $\nu^{\mathcal L}\{\underline{\underline r}:
s_\varepsilon(\underline{\underline r}) \leq \delta\}>0$ almost
surely, for all $\delta>0$. On the other hand, since $\mathcal L$ is
compact, there is a random variable $\Delta>0$ such that
$\nu^{\mathcal L}\{\underline{\underline r}:
s_\varepsilon(\underline{\underline r}) \leq \Delta\}=0$, almost
surely by Corollary~\ref{cor:compact}. In particular, there is
$\delta>0$ such that $$\nu^{\mathcal L}\{\underline{\underline r}:
s_\varepsilon(\underline{\underline r}) \leq \delta\}=0$$ with
positive probability, which gives a contradiction.

Last, assume that $\mathcal L$ does not come down from infinity and
recall that $\Lambda$ cannot have an atom at $0$ in this case. It has
been shown in Proposition~23 of \cite{Pit1999} that the total
coalescence rate of all lines is infinite for all times, almost
surely. This is easy to see from the construction of
$\Lambda$-coalescence using the Poisson process with intensity
$\Lambda(dx)/x^2$, since the total coalescence rate of the partition
element containing 1, given that there are infinitely many lines, is
\begin{align*}
  \int_0^1 x \frac{\Lambda(dx)}{x^2} = \int_0^1 x^{-1} \Lambda(dx) =
  \infty,
\end{align*}
since the dust-free property, \eqref{eq:df}, holds by assumption.

Let $0<\eta<\delta$ and consider the $\delta$-ball around $1$ in
$L^\Pi$. Since the coalescence rate is infinite and an infinite number
of lines coalesce to the line containing 1 between times $\eta$ and
$\delta$, there is an infinite $\eta$-separated set in
$B_\delta(\{1\})$. Hence,
\begin{align*}
  \nu^{\mathcal L}\{\underline{\underline r}:
  \xi_\eta(\tau_\delta(\underline{\underline r}))<\infty\}=0,
\end{align*}
almost surely. Hence, for any sequences $0<\eta_n<\delta_n$ with
$\delta_n\xrightarrow{n\to\infty}0$, we find that
\begin{align*}
  \nu^{\mathcal L}(\bigcap_{0<\eta<\delta_n} \{\underline{\underline
    r}: \xi_\eta(\tau_{\delta_n}(\underline{\underline r}))<\infty\}
  \leq \nu^{\mathcal L}(\{\underline{\underline r}:
  \xi_{\eta_n}(\tau_{\delta_n}(\underline{\underline r}))<\infty\} =0,
\end{align*}
almost surely. By Theorem~\ref{T:loccompact}, $\mathcal L$ cannot be
locally compact.


\begin{thebibliography}{19}
\providecommand{\natexlab}[1]{#1}
\providecommand{\url}[1]{\texttt{#1}}
\providecommand{\urlprefix}{URL }
\expandafter\ifx\csname urlstyle\endcsname\relax
  \providecommand{\doi}[1]{doi:\discretionary{}{}{}#1}\else
  \providecommand{\doi}{doi:\discretionary{}{}{}\begingroup
  \urlstyle{rm}\Url}\fi
\providecommand{\eprint}[2][]{\url{#2}}

\bibitem[{Addario-Bery et~al.(2010)Addario-Bery, Broutin and
  Goldschmidt}]{ABBG09}
L.~Addario-Bery, N.~Broutin and C.~Goldschmidt.
\newblock The continuum limit of critical random graphs.
\newblock \emph{Probab. Theory Relat. Fields} \textbf{online first} (2010).

\bibitem[{Aldous(1993)}]{Ald1993}
D.~Aldous.
\newblock The continuum random tree {III}.
\newblock \emph{Ann. Probab.} \textbf{21}~(1), 248--289 (1993).

\bibitem[{Aldous(1999)}]{Aldous1999}
D.~Aldous.
\newblock Deterministic and stochastic models for coalescence (aggregation and
  coagulation): a review of the mean-field theory for probabilists.
\newblock \emph{Bernoulli} \textbf{5}~(1), 3--48 (1999).

\bibitem[{Berestycki(2009)}]{Berestycki01}
N.~Berestycki.
\newblock Recent progress in coalescent theory.
\newblock \emph{Sociedade Brasileira de Matem$\acute{a}$tica, Ensaios
  Matem$\acute{a}$ticos Volume 16, 1-193}  (2009).

\bibitem[{Bertoin and {Le~Gall}(2006)}]{BertoinLeGall2006}
J.~Bertoin and J.-F. {Le~Gall}.
\newblock Stochastic flows associated to coalescent processes {III:} limit
  theorems.
\newblock \emph{Illinois J. Math} \textbf{50}, 147--181 (2006).

\bibitem[{Burago et~al.(2001)Burago, Burago and Ivanov}]{BurBurIva01}
D.~Burago, Y.~Burago and S.~Ivanov.
\newblock A course in metric geometry, graduate studies in mathematics.
\newblock \emph{AMS, Boston, MA} \textbf{33} (2001).

\bibitem[{Evans(2000)}]{Evans2000}
S.~Evans.
\newblock Kingman's coalescent as a random metric space.
\newblock In \emph{Stochastic Models: Proceedings of the International
  Conference on Stochastic Models in Honour of Professor Donald A. Dawson,
  Ottawa, Canada, June 10-13, 1998 (L.G Gorostiza and B.G. Ivanoff eds.)},
  Canad. Math. Soc. (2000).

\bibitem[{Evans and O'Connell(1994)}]{EvansOConnell1994}
S.~Evans and N.~O'Connell.
\newblock Weighted occupation time for branching particle systems and a
  representation for the supercritical superprocess.
\newblock \emph{Canad. Math. Bull.} \textbf{37}~(2), 187--196 (1994).

\bibitem[{Evans et~al.(2006)Evans, Pitman and Winter}]{EvaPitWin2006}
S.~Evans, J.~Pitman and A.~Winter.
\newblock Rayleigh processes, real trees, and root growth with re-grafting.
\newblock \emph{Prob. Theo. Rel. Fields} \textbf{134}~(1), 81--126 (2006).

\bibitem[{Greven et~al.(2009)Greven, Pfaffelhuber and Winter}]{GPWmetric09}
A.~Greven, P.~Pfaffelhuber and A.~Winter.
\newblock {Convergence in distribution of random metric measure spaces (The
  $\Lambda$-coalescent measure tree)}.
\newblock \emph{Probab. Theory Relat. Fields} \textbf{145}~(1), 285--322
  (2009).

\bibitem[{Greven et~al.(2010)Greven, Pfaffelhuber and
  Winter}]{GrevenPfaffelhuberWinter2010}
A.~Greven, P.~Pfaffelhuber and A.~Winter.
\newblock Tree-valued resampling dynamics (martingale problems and
  applications).
\newblock \emph{Submitted}  (2010).

\bibitem[{Gromov(1999)}]{Gromov}
M.~Gromov.
\newblock \emph{Metric structures for Riemannian and Non-Riemannian spaces}.
\newblock Birkh\"{a}user, Basel (1999).

\bibitem[{Kingman(1982)}]{Kingman1982a}
J.F.C. Kingman.
\newblock The coalescent.
\newblock \emph{Stochastic Process. Appl.} \textbf{13}~(3), 235--248 (1982).

\bibitem[{{Le~Gall}(1999)}]{MR1714707}
J.-F. {Le~Gall}.
\newblock \emph{Spatial branching processes, random snakes and partial
  differential equations}.
\newblock Lectures in Mathematics ETH Z\"urich. Birkh\"auser Verlag, Basel
  (1999).

\bibitem[{{Le~Gall}(2007)}]{LeGall2007}
J.-F. {Le~Gall}.
\newblock The topological structure of scaling limits of large planar maps.
\newblock \emph{Invent. Math.} \textbf{169}, 621--670 (2007).

\bibitem[{Pitman(1999)}]{Pit1999}
J.~Pitman.
\newblock Coalescents with multiple collisions.
\newblock \emph{Ann. Prob.} \textbf{27}~(4), 1870--1902 (1999).

\bibitem[{Schweinsberg(2000{\natexlab{a}})}]{Schweinsberg2000a}
J.~Schweinsberg.
\newblock Coalescents with simultaneous multiple collisions.
\newblock \emph{Elec. J. Prob.} \textbf{5}~(12), 1--50 (2000{\natexlab{a}}).

\bibitem[{Schweinsberg(2000{\natexlab{b}})}]{Schweinsberg2000}
J.~Schweinsberg.
\newblock A necessary and sufficient condition for the {$\Lambda$}-coalescent
  to come down from infinity.
\newblock \emph{Elec. Comm. Prob.} \textbf{5}, 1--11 (2000{\natexlab{b}}).

\bibitem[{Vershik(1998)}]{Ver1998}
A.~M. Vershik.
\newblock The universal {U}rysohn space, {G}romov metric triples and random
  matrices on the natural numbers.
\newblock \emph{Russian Math. Surveys} \textbf{53}~(3), 921--938 (1998).

\end{thebibliography}

\end{document}